\documentclass{amsart}
\usepackage{amssymb}
\usepackage{amsmath,amsfonts,amsthm,amstext}

\newtheorem{theorem}{Theorem}[section]
\newtheorem{lemma}[theorem]{Lemma}

\newtheorem{cor}[theorem]{Corollary}
\newtheorem{proposition}[theorem]{Proposition}
\numberwithin{equation}{section}

\newcommand{\G}{\mathcal{G}}
\newcommand{\U}{\mathcal{U}}

\def\Vightarrow#1{\smash{\mathop{\longrightarrow}\limits^{#1}}}

\begin{document}

\title{Subgroups and homology of extensions of centralizers of pro-p groups}

\author{D. H. Kochloukova and P.A. Zalesskii}

\thanks{Partially supported by
FAPESP, Brazil.}

\maketitle

\section{Introduction}

Limit groups have been studied extensively over the last ten years
and they played a crucial role in the solution of the Tarski
problem. The name \emph{limit group} was introduced by Sela. There
are different equivalent definitions for these groups. The class
of limit groups coincides with the class of fully residually free
groups; under this name they were studied by Remeslennikov,
Kharlampovich and Myasnikov. One can also define limit groups in a
constructive manner as finitely generated subgroups of groups
obtained from free groups of finite rank by finitely many
extensions of centralizers. Starting from this definition,  a
special class $\mathcal{L}$ of pro-$p$ groups  (pro-$p$ analogues
of limit groups) was introduced by the authors in \cite{KZ}. The class
$\mathcal{L}$ consists of all finitely generated subgroups of
pro-$p$ groups obtained from free pro-$p$ groups of finite rank by
finitely many extensions of centralizers (see Subsection
\ref{sectionL},  for details). In \cite{KZ} it was shown that many
properties that hold for limit groups are also satisfied by the
pro-$p$ groups from the class $\mathcal{L}$ : the groups $G$ of the
class $\mathcal L$ have finite cohomological dimension, are of
type $FP_{\infty}$, have non-positive Euler characeristic, if $C$
is maximal procyclic subgroup then its centralizer and normalizer
coincide, every finitely generated normal subgroup has finite
index, every soluble $G$ is abelian, every two generated $G$ is
free or abelian. All these properties are known to hold for
abstract limit groups but the similarity might end there as it is
not known whether the groups from the class $\mathcal L$ are fully
residually free pro-$p$ or whether the pro-$p$ completion $D_d$ of
an orientable surface group of odd  genus $d$ and $p \not= 2$ is
in $\mathcal L$. Still it was shown by the authors in \cite{KZ2} that for odd $d$
the pro-$p$ group $D_d$ is in the class $\mathcal L$ and $D_d$ is
fully residually free.

The study of the groups of the class $\mathcal L$ was continued by Snopche, Zalesskii in
\cite{PavelIlir} where it was shown that a group from the class
$\mathcal L$ has Euler characteristic $\chi(G) = 0$ if and only if
$G$ is abelian. A profinite version of the class $\mathcal L$ was
considered by Zapata in his PhD thesis \cite{Theo}.

In the present paper we study asymptotic behavior of the
dimensions of the homology groups of subgroups of finite index in
$G \in \mathcal L$. In the case of abstract limit groups this was
recently studied by Bridson, Kochloukova in \cite{BK} where the case $\cap_i U_i = 1$ and
each $U_i$ normal in $G$  was considered. In the case of abstract
groups this is important as permits the calculation of the
analytic Betti numbers of abstract limit groups. In the case of
pro-$p$ groups little is known about such limits, if dimension one
homologies are used it is sometimes referred to as the rank
gradient of a pro-$p$ group.

\medskip
{\bf Theorem A} {\it Let $G \in {\mathcal L}$ and $\{ U_i \}_{i
\geq 1}$ be a sequence of open subgroups of $G$  such that
$U_{i+1}  \leq U_i$ for all $i$ and $cd(\cap_i U_i) \leq 2$. Then
\begin{itemize}
\item[(i)]
$$
\lim_{\overrightarrow{~i~}} \dim H_j(U_i, \mathbb{F}_p) / [G :
U_i] = 0 \hbox{ for }j \geq 3;$$

\item[(ii)]  if $[G : U_i]$ tends to infinity  and def$(U_i)$
denotes the deficiency we have
$$
\lim_{\overrightarrow{~i~}} def(U_i) / [G : U_i] = - \chi(G);$$
\item[(iii)] if $[G : U_i]$ tends to infinity and  $(\cap_i U_i) =
1$ we have
$$
\lim_{\overrightarrow{~i~}} \dim H_2(U_i, \mathbb{F}_p) / [G :
U_i] = 0 \hbox{ and } \lim_{\overrightarrow{~i~}}  \dim H_1(U_i,
\mathbb{F}_p) / [G : U_i] =  - \chi(G).$$\end{itemize}}

This direction of study lead us to prove  group theoretic
structure properties of the pro-$p$ groups from the class
$\mathcal{L}$ that are of independent interest.
  For a pro-$p$ group $G$ we denote the minimal number of generators
by $d(G)$. The following result shows that for an open subgroup $U$ of
$G$, the function $d(U)$ is a monotonic increasing function of
index $[G:U]$. We do not know whether this holds for abstract
limit groups.

\medskip
{\bf Theorem B} {\it Let $G \in {\mathcal L}$ be non-abelian and
$U$ be a  normal open subgroup of $G$ of index $p$. Then $d(U) >
d(G)$.}

\medskip
Theorem B is essentially used in the proof of the following result.

\medskip
{\bf Theorem C} {\it Let $N \leq H \leq G$, with   $G \in
{\mathcal L}$ non-abelian or $G$ is non-trivial free pro-$p$
product, $H$ finitely generated and $N$ normal in $G$. Then $[G :
H] < \infty$.}

\medskip
In the case of abstract limit groups Theorem C was proved by Bridson, Howie, Miller, Short in \cite{BHMS}. It is worth noting that
in the abstract case the full strength of Bass-Serre theory was
used and in the pro-$p$ case some part of this technique  is not
available. In particular geometric methods in the pro-$p$ case failed, whereas the proof in
\cite{BHMS} is purely geometric.

\section{Preliminaries}

\subsection{Pro-$p$ groups acting on pro-$p$ trees} \label{proptree}

One of the main tools we use in this paper is the action for
pro-$p$ groups on pro-$p$ trees. By definition  a pro-$p$ tree $\Gamma$ is a profinite graph, where for the pointed space $(E^*(\Gamma), *)$ defined by  $\Gamma / V(\Gamma)$ with the image of $V(\Gamma)$ as a distinguished point $*$ we have a short exact sequence
$$
0 \to [[\mathbb{F}_p((E^*(\Gamma), *)]] \Vightarrow{\delta} [[\mathbb{F}_p V(\Gamma)]] \Vightarrow{\epsilon} \mathbb{F}_p \to 0
$$
where $\delta(\bar{e}) = d_1(e) - d_0(e)$, $d_0, d_1 : E(\Gamma) \to V(\Gamma)$ define  beginning and end of an edge $e \in E(\Gamma)$, $\epsilon(v) = 1$
for every $v \in V(\Gamma)$ and $\delta(*) = 0$. We shall  assume that the action of a pro-$p$  group $G$ on a pro-$p$ tree $\Gamma$ is always continuous.
We denote by $\widetilde{G}$  the closed subgroup of $G$ generated by all vertex stabilizers $G_v$ for $v \in V(\Gamma)$. Then by \cite{RZ2} the
pro-$p$ group $G / \widetilde{G}$ acts freely on the pro-$p$ tree $\Gamma/ \tilde G$, hence $G / \widetilde{G}$ is free pro-$p$.

We note that if a pro-$p$ group
$G$ acts on a pro-$p$ tree $T$ cofinitely then we have the
decomposition of $G$ as the fundamental group of a graph of
pro-$p$ groups $(\G, T/ G)$ i.e. with the underlying graph $T / G$
and vertex and edge groups that are some vertex and edge
stabilizers of the action of $G$ on $T$. By
\cite[Prop.~4.4]{Z-M90} if $U$ is an open subgroup of $G$ then the
action of $U$ on $T$ is cofinite and so $U$ is the fundamental
group of the graph of pro-$p$ groups $(\U, T/ U)$  i.e. with
underlying graph $T / U$. Furthermore as in the Bass-Serre theory
of abstract groups   the vertex groups in $(\U, T/ U)$ are $G_v^g
\cap U$ where $v$ runs through the vertices of $T/ G$ and $g \in
G_v \setminus G / U$ and the edge groups are $G_e^g \cap U$ where
$e$ runs through the edges of $T/ G$ and $g \in G_e \setminus G /
U$.

Another important fact  is that  a pro-$p$  group $G$ acting on a
pro-$p$ tree with trivial edge stabilizers and $H$ is a pro-$p$
subgroup, possibly of infinite index, then results from the
classical Bass-Serre theory work in their pro-$p$ version. The
following result follows from the proof of \cite[Thm. 3.6]{Pavel};
see also the last section of \cite{melnikov}.

\begin{theorem} \label{kurosh} Let $G$ be a pro-$p$ group  acting on a pro-$p$ tree $T$ with
trivial edge stabilizers such that there exists a continuous
section $ \sigma : V (T) / G  \to V (T)$. Then $G$ is isomorphic
to a free pro-p product
$$
(\coprod_{v \in V(T) / G} G_{\sigma(v)}) \coprod ( G / \langle G_w
\mid w \in V(T) \rangle
$$
\end{theorem}

\subsection{The class of pro-$p$ groups ${\mathcal L}$ : extensions of centralizers} \label{sectionL}

The class of pro-$p$ groups $\mathcal L$ was defined in \cite{KZ}.
The groups in the class $\mathcal L$ are  finitely generated
pro-$p$ subgroups of $G_n$, $n \geq 0$, where $G_0$ can be any
finite rank free pro-$p$ group and for $n \geq 1$ we have  $G_n =
G_{n-1} \coprod_C (C \times \mathbb{Z}_p^s)$ for some $s \geq 1$,
where $C$ can be any self-centralized pro-$p$ subgroup in
$G_{n-1}$.

 Let $G \in {\mathcal L }$. Then $G \leq G_n$, where $n$ is the height of $n$ i.e. is the smallest $n$ possible. Then $G_n$ acts on the pro-$p$ tree $T$
 associated to the decomposition $G_n = G_{n-1} \coprod_C (C \times \mathbb{Z}_p^s)$. By restriction $G$ acts on $T$ but in general $T/ G$ is not finite and
 hence applications of the pro-$p$ version of Bass-Serre theory for pro-$p$ groups is difficult for this
 action.
 Still in  \cite{PavelIlir} Snopche, Zalesskii proved that $T$ can be modified to a pro-$p$ tree $\Gamma$ with cofinite $G$-action. We shall use this  several times in the paper and therefore
 state in the form we need.

 \begin{theorem}\label{cofiniteaction} There is a pro-$p$ tree $\Gamma$  on which $G$ acts cofinitely (i.e. $\Gamma/G$ is finite) such that the vertex stabilizers
 and the non-trivial edge stabilizers of the
 action of $G$ on
 $\Gamma$ are  vertex and edge stabilizers  of the action of $G$ on $T$. Thus $G=\pi_1(\G, \Gamma/G)$ is the fundamental group
 of pro-$p$ groups, where  vertex and edge groups are certain stabilizers of vertices and edges of $\Gamma$ in $G$.\end{theorem}

\section{Proof of Theorem B}

We start with a simple lemma about finite $p$-groups.

 \begin{lemma} \label{X} Let $K$ be a finite $p$-group with a normal subgroup $A \simeq \mathbb{F}_p^d$ and $[K : A] = p$. Let
 $y,z \in K\setminus A$  such that $ \Phi(K) z = \Phi(K)  y$. Then $z^p = y^p$ and the elements of the set
 $\{ (z y^{-1})^{x^j} \}_{0 \leq j \leq p-1}$ are not linearly independent (over $\mathbb{F}_p$) for any $x \in K \setminus A$.
\end{lemma}

\begin{proof}  Note that since $ \Phi(K) z = \Phi(K)  y$ we have that  $z = y c$,  $c \in \Phi(K)$. We view $A$ as a right  $\mathbb{F}_p[K/A] $-module  via
conjugation, denote by upper right index the action of $\mathbb{F}_p[K/A]$
on $A$.  Then $c \in \Phi(K) = [K,K] K^p = \langle A^I, z^p
\rangle \leq A$ where $I$ is the augmentation ideal of  of
$\mathbb{F}_p[K/A]$.

Note that $[K, K] = \langle a^{z-1} | a \in A \rangle$ and $(a
z^j)^p = a^{1 + z + \ldots + z^{p-1}} z^{j p}$ for $1 \leq j \leq
p-1$. Then $y^p = (zc )^p = z^p c^{1 + z + \ldots + z^{p-1}}$. We
use $\bar z$ for the image of $z$ in $K/A$. Since the augmentation
ideal $I$ annihilates the socle of $\mathbb{F}_p[K/A]$,  the
element $c^{1 + z + \ldots + z^{p-1}} \in \Phi(K)^{(1 + \bar
z + \ldots + \bar z^{p-1})}$ is the trivial element, hence $y^p =
z^p$.

Since $(z y^{-1})^{z^j} = (z c z^{-1})^{z^j} = c^{z^{j-1}}$ and
$c^{1 + z + \ldots + z^{p-1}} $ is the trivial element,  the
elements of the set $\{ (z y^{-1})^{z^j} \}_{0 \leq j \leq p-1}$
are not linearly independent (over $\mathbb{F}_p$). Finally note
that $\{ (z y^{-1})^{z^j} \}_{0 \leq j \leq p-1}= \{ (z
y^{-1})^{x^j} \}_{0 \leq j \leq p-1}$.
\end{proof}

\begin{lemma} \label{generation}  a) Let $A = G_1 \coprod_C G_2$ be  an amalgamated free pro-$p$ product where $C$ is procyclic. Then $$d(G_1) + d(G_2) - 1 \leq d(A) \leq d(G_1) + d(G_2)$$ and $d(A) = d(G_1) + d(G_2)$ if and only if $C \subseteq \Phi(G_1) \cap \Phi(G_2)$.

b) Let $A = HNN(G_1, C, t) = \langle G_1, t | c^t = c_1 \rangle$
be a proper pro-$p$ HNN extension with procyclic associated subgroups
$\langle c \rangle  = C \simeq C_1 = \langle c_1 \rangle$ and
stable letter $t$. Then $$d(G_1) \leq d(A) \leq d(G_1) + 1$$ and
$d(A) = d(G_1) + 1$ if and only if $ c \Phi(G_1) = c_1 \Phi(G_1)$.

c) Let $A$ be a pro-$p$ group built from the pro-$p$ group $B$
after applying finitely many pro-$p$ HNN extensions with stable
letters $t_1, \ldots , t_k$ and associated pro-$p$ subgroups $C_i
= \langle c_i \rangle$ and $\widetilde{C}_i = \langle \tilde{c}_i
\rangle$ for $1 \leq i \leq k$, i.e. $c_i^{t_i} = \tilde{c}_i$.
Assume further that $C_i \cup \widetilde{C}_i \subseteq B$ for all
$1 \leq i \leq k$. Then $$d(B) \leq d(A) \leq d(B) + k.$$
Furthermore

1. $d(A)  =  d(B) + k$ if and only if    $ c_i \Phi(B) =
\tilde{c_i} \Phi(B)$ for all  $1 \leq i \leq k$;

2.  $d(A)  =  d(B)$ if and only if  the images of the elements $\{
c_i^{-1} \tilde{c}_i \}_{1 \leq i \leq k}$  in $B / \Phi(B)$ are
linearly independent (over $\mathbb{F}_p$).
\end{lemma}

\begin{proof} Parts a) and b) follow directly from the fact that for a pro-$p$ group $G$ we have $d(G) = \dim_{\mathbb{F}_p} G/\Phi(G)$, where $\Phi(G) =  G' G^p$. Note that by \cite{Ribes} pro-$p$ amalgamated products over procyclic subgroup are proper. Part c) follows from part b) by induction on $k$.
\end{proof}
\begin{theorem} \label{growing} Let $G \in {\mathcal L}$  be non-abelian and $U$ be a  normal open subgroup of $G$ of index $p$. Then  $d(U) > d(G)$.
\end{theorem}

\begin{proof}
 Let $G \leq G_n$, where $n$ is the height of $G$. Then by Theorem \ref{cofiniteaction} $G=\Pi_1(\G,\Gamma)$ is the fundamental group of a finite connected graph of groups with procyclic edge groups and
 with vertex groups from the class $\mathcal L$ all of smaller height. Inducting on the height $n$ of $G$ we can assume that the theorem holds for all vertex groups.
 Furthermore we can induct on the size of  $\Gamma$.

 Assume that $e$ is an edge of $\Gamma$. Note that the edge stabilizer $C$ of $e$ is either infinite procyclic or trivial.  Then $G$ is either a free pro-$p$ product with amalgamation  $G_1 \coprod_C G_2$  or a pro-$p$ HNN extension $\langle G_1 , t | C^t = C_1 \rangle$. By induction the result holds for both $G_1$ and $G_2$.

1. Assume that $G = G_1 \coprod_C G_2$. Note that if $C \simeq
\mathbb{Z}_p$ then at least one of $G_1$ and $G_2$ is not abelian.
Indeed if not $\chi(G) = \chi(G_1) + \chi(G_2) - \chi(C) = 0 + 0 -
0 =  0$ and by \cite[Prop.~3.4]{PavelIlir} $G$ is abelian, a
contradiction.

Let $T$ be the standard pro-$p$ tree on which $G$ acts i.e. $T/G$
has one edge and two  vertices, the edge stabilizers are
conjugates of  $C$ and the vertex stabilizers are conjugates of
$G_1$ or of $G_2$. By restriction $U$ acts on $T$ and since $[G :
U] < \infty$ we can decompose $U$ as the fundamental group
 of a finite graph of pro-$p$ groups  $(\U,T/U)$ with underlying graph $T / U$ (see Proposition 4.4 in \cite{Z-M90} or subsection \ref{proptree} in the preliminaries).
Then we have three cases.

1.1. Suppose that $C U = G$. Then $C \not= 1$, $U G_1 = G = U G_2$
and $T/U$ has one edge with two  vertices. The vertex groups of
the graph of groups $(\U,T/U)$ are $U \cap G_1$ and $U \cap G_2$
and the edge group is $U \cap C$. Thus $U$ is the proper pro-$p$
amalgamated product  $(U \cap G_1) \coprod_{C \cap U} (U \cap
G_2)$.  By symmetry we can assume that $G_1$ is not abelian but
$G_2$ might be abelian. Then by induction $d(U \cap G_1) \geq 1 +
d(G_1)$ and $d(U \cap G_2) \geq d(G_2)$, so
$$
d(U) \geq d(U \cap G_1) + d(U \cap G_2) - 1 \geq d(G_1) + 1 +
d(G_2) -1 = d(G_1) + d(G_2)
 \geq d(G).$$
But if $d(U) = d(G)$ then $d(G) = d(G_1) + d(G_2)$, so by Lemma
\ref{generation} $C \subseteq \Phi(G_1) \cap \Phi(G_2)\leq \Phi
(G_1) \subseteq U$, in particular $C \subseteq U$ a contradiction.
Thus $d(U)>d(G)$.

1.2. Suppose that $G_2 \subseteq U$ and $G_1 U = G$. Then $U$ is
the fundamental group of the graph of groups  $(\U,T/U)$ with
vertex groups $G_1 \cap U$ and $G_2, G_2^x, \ldots,
G_2^{x^{p-1}}$, where $x \in G \setminus U$, so we can assume that
$x \in G_1 \setminus U$. The edge group linking the vertex groups
$G_1 \cap U$ and $G_2^{x^i}$ is $C^{x^i}$. Then
$$
d(U) \geq d(G_1 \cap U) + p. d(G_2) - p.
$$
Note that $C$ is a direct factor in its  centralizer in $G_2$ and
this centraliser is  a finitely generated abelian group (it might
be procyclic). Thus if $d(G_2) = 1$ we get that $C = G_2$ and $G =
G_1$ and by induction the lemma holds for $G_1$, so that we are
done. Then we can assume $d(G_2) \geq 2$.

Assume first that $G_1$ is not abelian, so by induction $d(G_1
\cap U) \geq d(G) + 1$. Then
$$
d(U) \geq d(G_1 \cap U) + p . d(G_2) - p \geq d(G_1) + 1 + p .
d(G_2) - p \geq d(G_1) + d(G_2) + 1 \geq d(G) + 1.
$$
Assume now that $G_1$ is abelian. Since $d(G_2) \geq     2$
$$
d(U) \geq d(G_1 \cap U) + p. d(G_2) - p = d(G_1) + p.( d(G_2) - 1)
\geq$$ $$ d(G_1) + 2 d(G_2) - 2 \geq d(G_1) + d(G_2) \geq d(G).
$$
If $d(G) = d(U)$ then $d(G_1) + d(G_2) = d(G)$ and by Lemma
\ref{generation} $C \subseteq \Phi(G_1) \cap \Phi(G_2) \subseteq
\Phi (G_1)$ contradicting $C$ being a direct summand in its
centraliser $C_{G_1}(C) = G_1$. Hence $d(U)>d(G)$.

1.3 Suppose that $C \subseteq U$, $G_2 U = G_1 U = G$. Then $U$ is
the fundamental group of the graph of groups $(\U,T/U)$  with two
vertex groups $U \cap G_1$ and $U \cap G_2$ and $p$ edges linking
this two vertices with edge groups $C, C^x, \ldots, C^{x^{p-1}}$,
where $x \in G \setminus U$.

If  at least one of $G_1$ and $G_2$ is non-abelian we can assume
(by symmetry) that $G_1$ is non-abelian. Then since $d(U \cap G_1)
\geq d(G_1) + 1$ and $d(U \cap G_2) \geq d(G_2)$
$$
d(U) \geq d(U \cap G_1) + d(U \cap G_2) + (p-1) - p \geq $$
$$d(G_1) + 1  + d(G_2) + (p-1) - p \geq d(G_1) + d(G_2) \geq d(G)
$$
If $d(U) = d(G)$ then $d(G) = d(G_1) + d(G_2)$ and $d(U \cap G_2)
= d(G_2)$,  hence by Lemma \ref{generation}  $C \subseteq
\Phi(G_1) \cap \Phi(G_2)$ and by induction hypothesis $G_2$ is abelian. But this
contradicts  the fact that
 $C$ is a direct summand of $G_2$. Thus $d(U)>d(G)$.

If both $G_1$ and $G_2$ are abelian then $C$ is trivial (since
$G\in {\mathcal L}$), hence
$$
d(U) \geq d(U \cap G_1) + d(U \cap G_2) + (p-1)  \geq d(G_1) +
d(G_2) + 1 > d(G) $$
 and this completes the case of $G$ splitting as an amalgamated free pro-$p$ product.

\medskip
2. Now assume that $G = HNN (G_1, C, t) = \langle G_1, t  | c^t =
c_1 \rangle$ i.e. $G$ is an HNN extension with a base group $G_1$
, stable letter $t$ and associated  procyclic subgroups $C = \langle
c \rangle$ and $C_1 = \langle c_1 \rangle$. If $C$ is trivial then
$G = G_1 \coprod \langle t \rangle$, hence we can apply the case
of amalgamated pro-$p$ products considered above. Thus we can
assume that $C$ is infinite.

If $G_1$ is abelian then $\chi(G)=\chi(G_1)-\chi(C) = 0$, so by
\cite[Thm.~3.4]{PavelIlir} $G$ is abelian, a contradiction. Thus
we can assume that $G_1$ is not abelian.

Let $T$ be the standard pro-$p$ tree on which $G$ acts i.e. $T/G$
has one edge (a loop), the edge stabilizers of the action of $G$
on $T$ are conjugates of  $C$ and  the vertex stabilizers of the
action of $G$ on $T$ are conjugates of  $G_1$. By restriction $U$
acts on $T$ and since $[G : U] < \infty$ we can decompose $U$ as
the fundamental group of a finite graph  of pro-$p$ groups
$(\U,T/U)$ with underlying graph $T / U$ (see Proposition 4.4 in
\cite{Z-M90} or subsection \ref{proptree} from the preliminaries). Then we have three cases.

2.1. Suppose that $U C = G$ , hence $G_1 U = G$. Then $U$ is the
fundamental group of the graph of groups $(\U,T/U)$ with one
vertex group $U \cap G_1$ and one edge group $U \cap C = C^p$. By
Lemma \ref{generation} $d(G_1) \leq d(G) \leq d(G_1) + 1$, where
$d(G) = d(G_1) + 1$ if and only if $c \Phi(G_1) = c_1 \Phi(G_1)$.
Similarly $d(G_1 \cap U) \leq d(U) \leq d(G_1 \cap U) + 1$. Since
$G_1$ is non-abelian by induction $d(G_1 \cap U) \geq d(G_1) + 1$.
Then
$$
d(U) \geq d(G_1 \cap U) \geq d(G_1) + 1 \geq d(G).
$$
If $d(U) = d(G)$ then $d(G_1) + 1 = d(G)$ and so by Lemma
\ref{generation} $ c \Phi(G_1) = c_1 \Phi(G_1)$. Also $d(U) =
d(G_1 \cap U) $ and $U \cap C = \langle c^p \rangle$, $U \cap C_1
= \langle c_1^p \rangle$,  hence by Lemma \ref{generation} we have
$c^p \Phi (G_1 \cap U) \not= c_1^p \Phi (G_1 \cap U)$.  This
contradicts Lemma \ref{X} for the group $K = G_1 / \Phi(U \cap
G_1)$. Thus $d(U)>d(G)$.

2.2. Suppose  $C \subseteq U$ and $G_1 U = G$. Then $U$ is the
fundamental group of a graph of groups $(\U,T/U)$ with one vertex
group $U \cap G_1$ and $p$ edge groups $C, C^x, \ldots,
C^{x^{p-1}}$, where $x \notin U$. Since $G_1$ is not abelian we
have $d(U \cap G_1) \geq d(G_1) + 1$, so by Lemma \ref{generation}
$$
d(U) \geq d(U \cap G_1) \geq d(G_1) + 1 \geq d(G).
$$
If $d(U) = d(G)$ then $d(G_1) + 1 = d(G)$ hence by Lemma
\ref{generation} $ c \Phi(G_1) = c_1 \Phi(G_1)$. As well $ d(U) =
d(U \cap G_1)$ then since  $U \cap C = \langle c \rangle$,
$U \cap C_1 = \langle c_1 \rangle$ we deduce by Lemma
\ref{generation} that $\{ (c^{-1} c_1)^{x^j} \Phi (U \cap G_1)
\}_{0 \leq j \leq p-1}$ are linearly independent in $U \cap G_1 /
\Phi(U \cap G_1)$, contradicting Lemma \ref{X}.

2.3. Suppose that $G_1 \subseteq U$. Then $U$ is the fundamental
group of the graph of groups $(\U,T/U)$ with  $p$ vertex groups
$G_1^{x^j} \cap U = G_1^{x^j}$ and $p$ edge groups $C^{x^j} \cap U
= C^{x^j}$ such that the underlying graph of groups is a circuit
of $p$ edges.  Then since $G_1$ is not abelian we have $d(G_1)
\geq 2$ and hence
$$
d(U) \geq  1 - p + \sum_{j=1}^p d(G_1^{x^j} \cap U) \geq 1 - p + p
d(G_1)  \geq d(G_1) + 1 \geq d(G).
$$
If $d(U) = d(G)$ then $p d(G_1) + 1 - p = d(G_1) + 1$, so $d(G_1)
= p / (p-1) \in \mathbb{Z}$, $p = 2$ and $d(G_1) = 2$. As well
$d(G_1) + 1 = d(G)$, so by Lemma \ref{generation} we have  $ c
\Phi(G_1) = c_1 \Phi(G_1)$.

Let $\bar{G}$ be the quotient group of $G$ modulo the normal
closure of $\Phi(G_1)$ in $G$. Then $\bar{G} = A \coprod_{\bar{C}}
(\bar{C} \times \langle \bar{t} \rangle)$, where $A = G_1 /
\Phi(G_1)$ and overlining stands for the image of a subgroup or an
element of $G$ in $\bar{G}$.

Assume first that $c \in \Phi(G_1)$. Then $\bar{G} = A \coprod
\langle \bar{t} \rangle$. Since $G_1 \subseteq U$ we deduce that
$A^{\bar{G}} \subseteq \bar{U}$, hence $\bar{U} = A \coprod
A^{\bar{t}} \coprod \langle \bar{t}^2 \rangle$. Thus $d(U) \geq
d(\bar{U}) = 2 d(A) + 1 = 2 d(G_1) + 1 = 5 > 3 =  d(G_1) + 1 =
d(G)$, a contradiction.

Now suppose that $c \notin \Phi(G_1)$. Define $K$ as the quotient
group of $\bar{G}$ by the central subgroup $\bar{C}$. Then $K = B
\coprod \langle f \rangle$, where $f$ is the image of $\bar{t}$ in
$K$ and $B = A / \bar{C}$. Since $B \subseteq W$, where $W$ is the
image of $\bar{U}$ in $K$, we get $ W = B \coprod B^f \coprod
\langle f^2 \rangle, $ so $d(W) = 2 d(B) + 1 = 3$. On other hand
since $\bar{C}$ is a direct factor in $\bar{G}$ we get that
$\bar{U} \simeq W \times \bar{C}$, hence $d(U) \geq d(\bar{U}) =
d(W) + 1 = 4 > 3 = d(G_1) +1 = d(G) = d(U)$, a contradiction.

\end{proof}

{\bf Remark.} We do not know whether Theorem \ref{growing} holds
in the abstract case.

\begin{cor}\label{product} Let $G \in \mathcal{L}$ be non-abelian. Let  $H$ be a
finitely generated of infinite index  and $C$ a procyclic subgroups
of $G$. Then $HC$ is not open in $G$.\end{cor}

\begin{proof} Suppose $HC$ is open in $G$. Then by going down to a subgroup of finite index in $G$ we can assume that $H C = G$, hence for every open subgroup $U$
of $G$ that contains $H$
 we have $d(U) \leq d(H) + 1 $ contradicting Therem
 \ref{growing}.\end{proof}

By Theorem \ref{cofiniteaction} a pro-$p$ group $G$ from class
$\mathcal{L}$ acts cofinitely on a pro-$p$ tree $\Gamma$. If $U_i$ is a
chain of open subgroups of $G$ then every $U_i$ is the fundamental
group of a graph of pro-$p$ groups over $\Gamma_i=\Gamma /U_i$.

\begin{proposition} \label{rankgrows} Let $G \in \mathcal{L}$ be non-abelian and let $\Gamma$ be a pro-$p$ tree on which $G$ acts cofinitely with procyclic edge
stabilizers. Let $H$ be a finitely generated subgroup of $G$ such
that $H/\widetilde H$ is not abelian. Then for every sequence $\{ U_i \}_{i \geq 1}$
of  open subgroups of $G$ such that $U_{i+1}$ is a normal subgroup
of index $p$ in $U_i$, $U_0 = G$ and $H=\bigcap_i U_i$ such that
$d(U_i / \widetilde{U_i})=rank(\pi_1(\Gamma/U_i))$ tends to infinity
when $i$ tends to infinity.
\end{proposition}

\begin{proof}  Let  $\{ U_i \}_{i \geq 0}$ be a sequence
of  open subgroups of $G$ such that $U_{i+1}$ is a normal subgroup
of index $p$ in $U_i$, $U_0 = G$ and $H=\bigcap_i U_i$.
 Let $\Gamma_i = \Gamma / U_i$. Thus $U_i=\Pi_1(\U_i,\Gamma_i)$ is the fundamental group of a graph of groups with vertex and edge groups being certain  vertex and
edge stabilizers  of the action of $U_i$ on $\Gamma$, so edge groups
are procyclic.

Observe that $U_i / \widetilde{U}_i$ are free pro-$p$ groups of rank $d(\pi_1(\Gamma_i))$ and  $\Gamma_{i+1} / (U_i / U_{i+1}) = \Gamma_i$ hence
there is a natural projection map $$\varphi_{i+1}:\Gamma_{i+1} \to
\Gamma_i$$ that induces a homomorphism $\varphi_{i+1}^* : \pi_1(\Gamma_{i+1}) \to \pi_1(\Gamma_i)$.
Note that by Schreier's formula
$d (U_{i+1} /
\widetilde{U_{i+1}}) \geq d(\overline{U}_{i+1}) \geq d(U_i / \widetilde{U_i})$, where $\overline{U}_{i+1}$ is the image of the map $U_{i+1} /
\widetilde{U_{i+1}} \to U_{i} /
\widetilde{U_{i}}$ induced by the incluion of ${U}_{i+1}$ in ${U}_i$. Since $H/\widetilde H$ is the inverse limit of
$U_i/\widetilde{U_i}$ we may assume that $U_i / \widetilde{U_i}
\simeq \pi_1(\Gamma_i)$ is non-procyclic free pro-$p$ group, in particular, $\Gamma_i$
is not a tree.

Assume that
$d(U_i / \widetilde{U_i})=rank(\pi_1(\Gamma/U_i))$ does not tend to infinity when $i$ tends to infinity. Then for some large $i_0$, for all $i \geq
i_0$, the ranks of $U_i / \widetilde{U_i}$ are the same, in particular since $U_i / \widetilde{U}_i$ is not procyclic we have  $\overline{U}_{i+1} = U_i / \widetilde{U}_i$, so the map $\varphi_{i+1}^*$ is an epimorphism ( hence an isomorphism).   From now on consider only $i \geq i_0$.

Note that if for a sufficiently large $i$, say $i \geq i_1$, we
have that $\Gamma_i = \Gamma_{i+1}$ then $T/ H = \Gamma_{i_1}$ is
finite. Then the number of edges of $T/ H$ is finite, hence we
have finitely many double coset classes $H \setminus G / C$ where
$C$ is some edge stabilizer of the action of $H$ on $T$.
 This implies that $H C$ is open in $G$ contradicting  Corollary \ref{product}. Thus we can assume that $\Gamma_{i+1} \not= \Gamma_i$ for infinitely many  $i$.

The main ingredient  of the proof is a description of how
$\Gamma_{i+1}$ and $\Gamma_i$ relate to each other when $\Gamma_{i+1} \not= \Gamma_i$. We recall that $U_{i+1}$ is the fundamental group of a graph of group over the graph $\Gamma_{i+1}$ and $U_{i}$ is the fundamental group of a graph of groups over the graph $\Gamma_{i}$ and the decomposition of $U_{i+1}$ as a graph of groups is induced by the decomposition of $U_i$ as a graph of groups as explained in the first paragraph of subsection \ref{proptree}.

We split $V(\Gamma_i)$ as a disjoint union $V_1 \cup V_2, $ where $U_i/ U_{i+1}$ fixes every element of  $ \varphi_{i+1}^{-1} (V_1)$
and $U_i / U_{i+1}$ acts freely on $ \varphi_{i+1}^{-1} (V_2)$.
Then the preimage of $V_1$ in $\Gamma_{i+1}$ has the same cardinality as $V_1$; the preimage of $V_2$ in
$\Gamma_{i+1}$ has cardinality $p|V_2|$. It follows that
$V_1\neq \emptyset$ since otherwise $\varphi_{i+1}$ is a covering
and so  $\varphi_{i+1}^*$ is not surjective contradicting to the
above.

Now there are 3 types of edges in $\Gamma_i$: 1) those that have
their vertices in $V_1$, 2) those that have their vertices in
$V_2$ and 3) those that have one of the vertices in $V_1$ and the
other in $V_2$. We denote them by $E_1$, $E_2$ and $E_{12}$
respectively.  Note that  $U_i/U_{i+1}$ acts freely on the preimages of $E_2$ and $E_{12}$ so that they have cardinalities $p|E_2|$ and $p|E_{12}|$.  Observe that since $\varphi_{i+1}$ is injective the preimage under $\varphi_{i+1}$ of an edge $e$ of $E_1$ contains exactly one edge , otherwise we can have a non-contractable closed path (of two edges) in $E_{i+1}$ whose image in $\Gamma_i$ is $e \bar{e}$, contradicting injectivity of $\varphi_{i+1}^*$.

Let $M_j$ be the graph obtained from $\Gamma_j$ contracting the connected components of the subgraph $\Delta = V_1 \cup E_1$ to a point, where $j = i$ or $j = i+1$. Assume that $\Delta$ has $k$ points. Since the restriction of $\varphi_{i+1}$ on $\Delta$ is a bijection the map $\pi_1(M_{i+1}) \to \pi_1(M_i)$ induced by $\varphi_{i+1}$ is an isomorphism.  Then
$$| E_2| + | E_{12} |  - k - | V_2| +1 = |E(M_{i})| - (V(M_i)| + 1 =
 rank(\pi_1(M_{i})) = $$ $$ rank(\pi_1(M_{i+1})) = |E(M_{i+1})| - (V(M_{i+1})| + 1  =
p| E_2| + p| E_{12} | - k - p| V_2| +1$$
Hence $(p-1) | E_2| + (p-1)| E_{12} |   = (p-1) |V_2|$ and $ 0 \leq rank(\pi_1(M_{i})) =  | E_2| + | E_{12} |  - k - | V_2| +1 = 1 - k \leq 0$, so $M_i$ and $M_{i+1}$  are trees, $\Delta$ is connected.
Thus $\Gamma_i$ is obtained from $\Delta$ by  attaching several trees and $\Gamma_{i+1}$ is obtained from $\varphi_{i+1}^{-1}(\Delta)$ by  attaching several trees, the restriction of $\varphi_{i+1}$ to $\varphi_{i+1}^{-1}(\Delta)$ is a bijection. Since the inverse limit $\Gamma$ of $\Gamma_i$ is a pro-$p$ tree we get that the fundamental group of $\Delta$ does not survive in the inverse limit, thus $\Delta$ and hence $\Gamma_i$ are simply connected, a contradiction.

\end{proof}

\begin{cor} Let $G \in \mathcal{L}$ be non-abelian and let $\Gamma$ be a pro-$p$ tree on which $G$ acts cofinitely with procyclic edge stabilizers. Let $H$ be a finitely generated subgroup of $G$ such that $H/\widetilde{H}$ is non abelian. Let
$\{ V_i \}_{i \geq 1}$ be open subgroups of $G$ such that
$V_{i+1}\leq V_i$, $V_0 = G$ such that $H=\bigcap_i V_i$.  Then
$d(V_i / \widetilde{V_i})$ tends to infinity when $i$ tends to
infinity.
\end{cor}

\begin{proof} We can refine the sequence $ \ldots V_{i+1} \leq V_i \leq \ldots \leq V_1 \leq G $ to get a sequence
$ \ldots U_{i+1} \leq U_i \leq \ldots \leq U_1 \leq G $ as in the
previous theorem and apply this theorem.
\end{proof}

\section{The core property}

 \begin{theorem} \label{freecore} Let $N \leq H \leq G$, where $G \in {\mathcal L}$ is not abelian, $H$ finitely generated and $N$ non-trivial normal in $G$.  Let $G$ act on a pro-$p$ tree $\Gamma$ with $\Gamma / G$ finite and procyclic edge stabilizers. Assume that $N$ acts freely on $\Gamma$. Then $[G : H] < \infty$.
\end{theorem}

\begin{proof} Suppose that $[G : H] = \infty$. Then by Theorem 6.5 \cite{KZ} $N$ is not finitely generated, we will need only that it is not procyclic.

By Proposition \ref{rankgrows}  $H / \widetilde{H}$ is the inverse limit of
$U / \widetilde{U}$ where $U$ runs through  a set $\{ U_i \}_{i \geq 0}$
of  open subgroups of $G$ such that $U_{i+1}$ is a normal subgroup
of index $p$ in $U_i$, $U_0 = G$ and $H=\bigcap_i U_i$ such that
$d(U_i / \widetilde{U_i})=rank(\pi_1(\Gamma/U_i))$ tends to infinity
when $i$ tends to infinity. Then $H_1(H/ \widetilde{H}, \mathbb{F}_p)$ is
the inverse limit of $H_1(U_i/ \widetilde{U_i}, \mathbb{F}_p)$ and
$H_1(H/ \widetilde{H}, \mathbb{F}_p)$ is finite. So for some $U_{i_0}$
we have that the inclusion map $ H \to U_{i_0}$ induces an injective
map $H_1(H/ \widetilde{H}, \mathbb{F}_p) \to H_1(U_{i_0} /
\widetilde{U_{i_0}}, \mathbb{F}_p)$. Hence for every $H \leq U_i \leq
U_{i_0}$  the map  $H_1(H/ \widetilde{H}, \mathbb{F}_p)
\to H_1(U_i / \widetilde{U_i}, \mathbb{F}_p)$ is injective. Since $H /
\widetilde{H}$ and $U_i / \widetilde{U_i}$ are  free pro-$p$ we get
that the map $H / \widetilde{H} \to U_i / \widetilde{U_i}$ is
injective and so $\widetilde{U_i}\cap H=\widetilde H$. Since $N$ is
infinitely generated (hence is not procyclic)  and acts freely on $T$  by replacing $H$
with some of its open subgroup containing $N$ we may assume that
$N\widetilde H/ \widetilde H$ is not abelian. Indeed if for every
 open subgroup $H_i$ of $H$ that contains $N$ we have that $H_i /
\widetilde{H_i}$ is procyclic then $N = N / \widetilde{N}$ is inverse
limit of the procyclic groups $H_i / \widetilde{H_i}$, so is procyclic,
a contradiction. Thus we may assume that $N\widetilde H/\widetilde H$ is
not abelian and in particular $H/\widetilde H$ is not abelian.
%Then  by Proposition \ref{rankgrows} $[G : H] = \infty$ implies
%that $rank(U / \widetilde{U})$ tends to infinity.

Consider $N \leq
U_{i_1} \leq U_{i_0}$ such that  $rank(H /
\widetilde{H}) < rank (U_{i_1} / \widetilde{U_{i_1}})$ and $N \not\subseteq
\widetilde{U_{i_1}}$. This is possible since $1 = \widetilde{N} = \cap_{i \geq i_0}
\widetilde{U_i}$  and  $rank(U_i / \widetilde{U_i})$ tends to infinity. Consider the groups
$$
N \widetilde{U_{i_1}} / \widetilde{U_{i_1}} \leq H / \widetilde{H} \leq U_{i_1} /
\widetilde{U_{i_1}}.
$$
Since  $N \widetilde{U_{i_1}} / \widetilde{U_{i_1}} $ is non-trivial normal
subgroup of a free pro-$p$ group $U_{i_1} / \widetilde{U_{i_1}}$,  $H /
\widetilde{H}$ is a subgroup of finite index in  $ U_{i_1} /
\widetilde{U_{i_1}}$ and  by Euler characteristic formula (Schreier
formula) $d(H / \widetilde{H}) - 1 = [U_{i_1} / \widetilde{U_{i_1}} : H /
\widetilde{H}] (d(U_{i_1} / \widetilde{U_{i_1}}) - 1)$ a contradiction with
$rank(H / \widetilde{H}) < rank (U_{i_1} / \widetilde{U_{i_1}})$.
\end{proof}

\begin{theorem}  \label{core111} Let $N \leq H \leq G$, with   $G \in {\mathcal L}$ non-abelian, $H$ finitely generated and $N$ non-trivial normal in $G$.
Then $[G : H] < \infty$.
\end{theorem}

\begin{proof}
Let $n$ be the height of $G$ i.e. $n$ is the smallest number such
that $G$ is a closed pro-$p$ subgroup of $G_n = G_{n-1} \coprod_C
(C \times B)$, where $C \simeq \mathbb{Z}_p$ is self-centralized
in $G_{n-1}$ and $B$ is $\mathbb{Z}_p^k$ for some $k \geq 1$. We
induct on $n$.

Let $T$ be the standard pro-$p$ tree on which $G_n$ acts i.e.
$T/ G_n$ has one edge and two vertices, the edge stabilizers are
conjugates of $C$ and the vertex stabilizers are conjugates of
$G_{n-1}$ and $C \times B$.

\medskip
{\bf Claim.} {\it There is a non-trivial normal subgroup $N_0$ of
$G$ such that $N_0 \leq N$ and $N_0$ acts freely on $T$.}

\medskip
{\it Proof of claim.}

If $\langle B \rangle^{G_n} \cap G = 1$ then $G$ embeds into
$G_{n-1}\cong \overline{G_n} = G_n / \langle B \rangle^{G_n}$
contadicting minimality of $n$. So $\langle B \rangle^{G_n} \cap G
\not= 1$. By Theorem 2.7 in  \cite{KZ}
$$
\langle B \rangle^{G_n} = \coprod_{g \in G_{n-1} / C} B^g
$$
so the commutator subgroup $K$ of $\langle B \rangle^{G_n} $ acts
freely on $T$. Since $G\cap \langle B \rangle ^{G_n}$ is not abelian (by Theorem
6.5 in \cite{KZ}), $K\cap G\neq 1$. If $N \cap (K\cap G) = 1$ then
in $G$ they generate $N \times (K\cap G) $ and so by commutative
transitivity it is abelian of rank at least 2. Then by Corollary
5.5 in \cite{KZ}  $N \times (K\cap G)$ is conjugate in $G_n$ into
$C \times B$, therefore is a finitely generated abelian normal
subgroup of $G$. But this is impossible because for any $g\in G
\setminus  (C\times B)$ we have $N \times (K\cap G)\leq (C\times
B)\cap (C\times B)^g$ is contained in a conjugate of $C$.
 Hence $N \cap K \not= 1$. Then it suffices to set $N_0 = N \cap K$. This completes the proof of the claim.

\medskip
Let $\Gamma$ be the pro-$p$ tree of Theorem \ref{cofiniteaction}. Then a subgroup of $G$ acts freely on
$\Gamma$ if and only if it acts freely on $T$. Thus the claim and
Theorem \ref{freecore} complete the proof of the theorem.
\end{proof}

\section{Aproximating homologies}

\begin{theorem}\label{general} Let $G$ be a profinite (pro-$p$) group acting on a profinite (pro-$p$) tree $T$ such that $T/G$ is finite and all vertex and
edge stabilizers are of type $FP_{\infty}$. Let $M$ be  a finite pro-$p$ $\mathbb{F}_p[[G]]$-module.  Let $\{ U_i \}_{i \geq 1}$ be a sequence of open
subgroups of $G$  such that for all $i$ we have $U_{i+1}  \leq U_i$ and
$$
\lim_{\overrightarrow{~i~}}  \dim H_j(U_i\cap G_v^g, M) / [G_v^g
:(G_v^g\cap U_i)] = \rho(v,g),$$ $$\lim_{\overrightarrow{~i~}}
\dim H_{j-1}(U_i\cap G_e^g, M) / [G_e^g :(G_e\cap U_i)] =
\rho(e,g),
$$ where $\rho(v,g), \rho(e,g)$ are
continuous functions with domains $V(T) \times G$ and $E(T) \times
G$ respectively.
Then  $$
\sup_{\overrightarrow{~i~}} \dim H_j(U_i, M) / [G : U_i] \leq
\sum_{v\in V(T)/G} sup_{g\in G}(\rho(v,g))+\sum_{e\in
E(T)/G}sup_{g\in G}(\rho(e,g)).
$$ In particular, if $\rho(v,g)$ and $\rho(e,g)$ are  the zero maps, then $$
\lim_{\overrightarrow{~i~}}  \dim H_j(U_i, M) / [G : U_i] =0.$$
\end{theorem}

\medskip \noindent
{\bf Remark.} If $U_i$ are normal then the functions $\rho(v,g)$
and $\rho(e,g)$ are constant on $g$.

\begin{proof}
Since $G_v$ and $G_e$ are all of type $FP_{\infty}$ all the groups
$H_j(U_i\cap G_v^g, M)$ and $H_j(U_i\cap G_e^g, M)$ are finite
dimensional over $\mathbb{F}_p$. Furthermore since $T/ G$ is
finite we deduce that $G$ is of type $FP_{\infty}$.

Consider the Mayer-Vietoris long exact sequence in homology for
the action of $G$ on $T$
$$
\ldots \to \bigoplus_{e \in E(T)/ G} \bigoplus_{g \in G_e
\setminus G / U_i}  H_j(U_i \cap G_e^g, M) \to  \bigoplus_{v \in
V(T)/ G} \bigoplus_{g \in G_v  \setminus G / U_i}  H_j(U_i \cap
G_v^g, M)$$ $$ \to H_j(U_i, M) \to \bigoplus_{e \in E(T)/
G}\bigoplus_{ g \in G_e  \setminus G / U_i}  H_{j-1}(U_i \cap
G_e^g, M) \to \ldots
$$ It follows that
\begin{equation} \label{dimdiseq} \dim H_j(U_i, M) \leq  \sum_{v \in V(T)/ G}\sum_{g \in G_v \ \setminus G / U_i}  \dim H_j(U_i \cap G_v^g, M) + $$ $$\sum_{e \in E(T)/ G}\sum_{ g \in G_e \ \setminus G / U_i}  \dim H_{j-1}(U_i \cap G_e^g, M) < \infty.
 \end{equation}
 By hypothesis
$$
\rho(i, v, g) = \dim H_j(U_i \cap G_v^g, M) / [G_v^g : U_i \cap
G_v^g]
$$
and
$$
\rho(i, e, g) = \dim H_{j-1}(U_i \cap G_e^g, M) / [G_e^g : U_i
\cap G_e^g]
$$
tend to $\rho(v,g)$ and $\rho(e,g)$ respectively as $i$ goes to
infinity. Since $G$ is  compact  for a fixed $v$ and $e$ the
sequences $\{ \rho(i,v,g ) \}_{i}$ and $ \{ \rho(i, e, g) \}_i$
tend respectively to $\rho(v,g)$ and   $\rho(e,g)$ {\bf
uniformly}. Hence for a fixed $\epsilon > 0$ there is $i_0$ such
that for $i \geq i_0$, $g \in G$ and $v \in V(T) / G, e \in E(T) /
G$ we have
\begin{equation} \label{diseq1}
H_j(U_i \cap G_v^g, M) \leq (\epsilon+\rho(v,g)) [G_v^g : U_i \cap
G_v^g]
\end{equation}
and
\begin{equation} \label{diseq2}
H_{j-1}(U_i \cap G_e^g, M) \leq (\epsilon+\rho(e,g)) [G_e^g : U_i
\cap G_e^g].
\end{equation}
Counting $G_v$ and $G_e$-orbits and the sizes of these orbits in
$G / U_i$ we have
\begin{equation} \label{cosetaction}
 \sum_{  g \in G_v
\setminus G / U_i}  [G_v^g : U_i \cap G_v^g] = [G : U_i] \hbox{
and } \sum_{  g \in G_e \setminus G / U_i}  [G_e^g : U_i \cap
G_e^g] = [G : U_i]
\end{equation}
Then by (\ref{dimdiseq}), (\ref{diseq1}), (\ref{diseq2}) and
(\ref{cosetaction}) we have
$$
\dim H_j(U_i, M) \leq \sum _{v \in V(T)/ G}\sum_{g \in G_v
\setminus G / U_i}  \dim H_j(U_i \cap G_v^g, M) + $$ $$\sum _{e
\in E(T)/ G}\sum_{g \in G_e  \setminus G / U_i}  \dim H_{j-1}(U_i
\cap G_e^g, M)\leq  \sum _{v \in V(T)/ G}\sum_{ g \in G_v
\setminus G / U_i}  (\epsilon+\rho(v,g)) [G_v^g : U_i \cap G_v^g]
+$$ $$ \sum _{e \in E(T)/ G}\sum_{ g \in G_e  \setminus G / U_i}
(\epsilon+\rho(e,g)) [G_e^g : U_i \cap G_e^g]\leq $$ $$ [G : U_i]
\big( \sum _{v \in V(T)/ G}(\epsilon+sup_g(\rho(v,g)))+\sum _{e
\in E(T)/ G} (\epsilon+ sup_g(\rho(e,g)))\big).
$$

Therefore for $i \geq i_0$
$$
\dim H_j(U_i, M) / [G : U_i] \leq \epsilon( | V(T/ G) | + |E(T/
G)| ) + $$ $$\sum_{v\in V(T)/G} sup_{g\in G}(\rho(v,g))+\sum_{e\in
E(T)/G}sup_{g\in G}(\rho(e,g)).$$ It follows that
$$\sup_{\overrightarrow{~i~}} \dim H_j(U_i, M) / [G : U_i] \leq
\sum_{v\in V(T)/G} sup_{g\in G}(\rho(v,g))+\sum_{e\in
E(T)/G}sup_{g\in G}(\rho(e,g)).
$$

\end{proof}

\begin{cor}\label{normal} Under the assumptions of Theorem \ref{general} with $j=1$ suppose further that $[G_e^g: (G_e^g\cap U_i)]$ tends to infinity for every $g \in G$ and  every fixed $e \in E(T)$ such that $G_e \not= 1$. Then if
$$
\lim_{\overrightarrow{~i~}} \dim H_1(U_i, M) / [G : U_i]$$ exists
we have
 $$
\lim_{\overrightarrow{~i~}} \dim H_1(U_i, M) / [G : U_i]
\leq\sum_{v\in V(T)/G} sup_{g\in G}(\rho(v,g)).$$\end{cor}

\begin{proof} Since $H_0(U,M)= M/J(U)M$, where $J(U)$ is the augmentation ideal of $\mathbb{Z}_p[[U]]$ and $U$ is a pro-$p$ group  we have
$$
\rho(e,g) = \lim_{\overrightarrow{~i~}} \dim  H_{0}(U_i\cap G_e^g,
M) / [G_e^g :(G_e^g \cap U_i)] = $$
$$\lim_{\overrightarrow{~i~}} \dim ( M / J(U_i\cap G_e^g) M )/
[G_e^g :(G_e^g \cap U_i)] \leq \lim_{\overrightarrow{~i~}} \dim M
/  [G_e^g :(G_e^g \cap U_i)] = 0.
$$
Thus we can apply Theorem \ref{general} to deduce the result.
\end{proof}
Recall that for a finitely presented pro-$p$ group $S$ the
deficiency $def(S) = \dim H_1(S, \mathbb{F}_p) - \dim H_2(S,
\mathbb{F}_p)$.

\begin{theorem}\label{>2} Let $G \in {\mathcal L}$ and $\{ U_i \}_{i \geq 1}$ be a sequence of open subgroups of $G$  such that $U_{i+1}  \leq U_i$ for all $i$ and
$cd(\cap_i U_i) \leq 2$. Then
\begin{itemize}
\item[(i)]
$$
\lim_{\overrightarrow{~i~}} \dim H_j(U_i, \mathbb{F}_p) / [G :
U_i] = 0 \hbox{ for }j \geq 3;$$ \item[(ii)]  if $[G : U_i]$ tends
to infinity we have
$$
\lim_{\overrightarrow{~i~}} def(U_i) / [G : U_i] = - \chi(G);$$
\item[(iii)] if $[G : U_i]$ tends to infinity and  $(\cap_i U_i) =
1$ we have
$$
\lim_{\overrightarrow{~i~}} \dim H_2(U_i, \mathbb{F}_p) / [G :
U_i] = 0 \hbox{ and } \lim_{\overrightarrow{~i~}}  \dim H_1(U_i,
\mathbb{F}_p) / [G : U_i] =  - \chi(G).$$\end{itemize}

\end{theorem}

\begin{proof} (i) We induct on the height of $G$. First if height of $G$ is 0, then $G$ is either free or abelian. If $G$ is free
$H_j(U_i, \mathbb{F}_p)  = 0$ for $j \geq 2$ and we are done. If
$G$ is abelian then $U_i \simeq G$ for every $i$, so $\dim
H_j(U_i, \mathbb{F}_p)  = \dim H_j(U_{i+1}, \mathbb{F}_p) $, hence
if $ \{ [G : U_i] \}_i$ tends to infinity we are done. If $ \{ [G
: U_i] \}_i$  does not tend to infinity then $U_i = U_{i+1} =
U_{i+2} = \ldots$, so $\cap_t U_t = U_i$ has finite index in $G$
and so $cd(G) = cd(\cap_j U_j) \leq 2$. Then $H_j(U_t,
\mathbb{F}_p) = 0$ for $j \geq 3$.

Assume that the theorem holds for groups from the class
$\mathcal{L}$ of smaller height. Let $n$ be the height of $G$ and
$G \subseteq G_n = G_{n-1} \coprod_{C_{n-1}} A_{n-1}$, where
$A_{n-1} = \mathbb{Z}_p^m = C_{n-1} \times B$. Then $G_n$ acts
cofinitely on a pro-$p$ tree $T$ with vertex stabilizers
conjugates of $G_{n-1}$ and $A_{n-1}$ and vertex stabilizers
conjugates of $C_{n-1}$. By Theorem \ref{cofiniteaction} $G$ acts
cofinitely on a pro-$p$ tree $\Gamma$ with vertex stabilizers that are
intersections of the vertex stabilizers of $G_n$ (in $T$)
with $G$ and edge stabilizers that are either trivial or infinite
cyclic.

Note that the height of $G_v$ is smaller than the height of $G$
since $G_v$ is inside of a conjugate of $G_{n-1}$ or a conjugate
of $A_{n-1}$. By induction applied for the group $G_v^g$ for a
fixed $j > 2$ we have
$$
\rho(i, v, g) = \dim H_j(U_i \cap G_v^g, \mathbb{F}_p) / [G_v^g :
U_i \cap G_v^g]
$$
tends to 0 as $i$ goes to infinity. Observe that since edge stabilizers are procyclic we have for any $j > 2$ that $H_{j-1}(U_i \cap G_e^g, \mathbb{F}_p)  = 0$, hence
$$
\rho(i, e, g) = \dim H_{j-1}(U_i \cap G_e^g, \mathbb{F}_p) / [G_e^g :
U_i \cap G_v^g] = 0.
$$ Then by Theorem \ref{general}
the result follows for $j>2$.

\smallskip
(ii) Since
$$
\chi(G) = \chi(U_i) / [G : U_i] = \sum_{j \geq 0} (-1)^j \dim
H_j(U_i, \mathbb{F}_p) / [G : U_i]$$ we have
$$
\chi(G) = \sum_{j \geq 3} \lim_{\overrightarrow{~i~}} \dim
H_j(U_i, \mathbb{F}_p) / [G : U_i]) - \lim_{\overrightarrow{~i~}}
def(U_i) / [G : U_i]  + \lim_{\overrightarrow{~i~}} 1/[G : U_i]$$
$$ = - \lim_{\overrightarrow{~i~}} def(U_i) / [G : U_i] $$

(iii)
 By induction applied for the group $G_v^g$
$$
\rho(v, g) = \lim_{\overrightarrow{~i~}}\dim H_2(U_i \cap G_v^g,
\mathbb{F}_p) / [G_v^g : U_i \cap G_v^g] = 0.
$$ Furthermore if $G_e \not= 1$
$$
0 \leq \rho(e, g) = \lim_{\overrightarrow{~i~}} \dim H_1(U_i \cap
G_e^g, \mathbb{F}_p) / [G_e^g : U_i \cap G_e^g] \leq
\lim_{\overrightarrow{~i~}} 1 / [G_e^g : U_i \cap G_e^g] = 0
$$
and if $G_e =1$ by the definition of $\rho(e,g)$ we have $\rho(e,g) = 0$.
 Then by Theorem \ref{general} $$
\lim_{\overrightarrow{~i~}}  \dim H_2(U_i, \mathbb{F}_p) / [G :
U_i] =0$$ and hence by (ii)
$$
\lim_{\overrightarrow{~i~}}  \dim H_1(U_i, \mathbb{F}_p) / [G :
U_i] = \lim_{\overrightarrow{~i~}}  \dim def(U_i) / [G : U_i] +
\lim_{\overrightarrow{~i~}}  \dim H_2(U_i, \mathbb{F}_p) / [G :
U_i]$$ $$ = - \chi(G) + 0 = - \chi(G).
 $$
\end{proof}

{\bf Remark.} The proof shows that (iii) also holds if we just
assume $[G_e^g:(U_i\cap G_e^g)]$ tends to infinity for every $e
\in E(T) / G$ such that $G_e \not= 1$ and for every $g\in G$.

\medskip
The Theorem \ref{>2} allows to obtain another proof of Theorem 6.5
in \cite{KZ} that we state as the following

\begin{cor} Let $H$ be a pro-p group from the class $\mathcal{L}$ with a non-trivial finitely generated
normal pro-$p$ subgroup $N$ of infinite index. Then $H$ is
abelian.\end{cor}

\begin{proof} We give two new proofs one of which is an application of the results from this section.

1. By   Proposition  13 in
\cite{Nick} if $N$ is a non-trivial finitely generated normal
subgroup of a finitely generated residually finite group $G$ of
infinite index, then rank gradient of $G$ is zero.  The pro-$p$
version of it is also valid with changing all the groups in the
proof to pro-$p$ groups. Applying this for $G = H$  and by Theorem \ref{>2} the rank gradient for pro-$p$ groups is
$$\lim_{\overrightarrow{~i~}}  \dim H_1(U_i,
\mathbb{F}_p) / [H : U_i] = - \chi(H).$$ Then $\chi(H) = 0$ and by \cite{PavelIlir} $H$ is abelian.

2. Note that the corollary is a particular case of  Theorem \ref{core111}. The only place its proof relies on Theorem 6.5
in \cite{KZ} is the fact that in Theorem \ref{freecore} $N$ cannot be procyclic. But if $N$ was procyclic then by \cite{EHKZ} by substituting $H$ with a subgroup of finite index we can assume that $H/N$ has finite cohomological dimension, thus $\chi(H/N)$ is well defined. Then $\chi(H) =  \chi(N) \chi(H/N) = 0$ and we are done as in the previous proof.\end{proof}


\begin{thebibliography}{99}

%\bibitem{Ardakov} K. Ardakov, Krull dimension of Iwasawa algebras, J. Algebra 280 (2004), 190 206


%\bibitem{BH} Bridson, Howie,  Subgroups of direct products of elementary free groups, Geom. Funct. Anal. 17 (2007), no. 2, 385 - –403.

\bibitem{Nick} M. Abert, A. Jaikin-Zapirain, N. Nikolov , The rank gradient from a combinatorial viewpoint, Groups Geometry and Dynamics, 2011, Vol. 5, p. 213-230.



\bibitem{BHMS}  M. R. Bridson, J.  Howie, C. F.   Miller III, H. Short,  Subgroups of direct products of limit groups.  Ann. of Math. (2)  170  (2009),  no. 3, 1447--1467

\bibitem{BK} M. Bridson, D. Kochloukova, Volume gradients and homology in towers of residually-free groups,
arxiv:1309.1877.

\bibitem{EHKZ} A. Engler, D. Haran, D. Kochloukova, P. Zalesskii,  Normal subgroups of profinite groups of finite cohomological dimension. J. of London Math. Soc., v. 69, (2004)  p. 317-33

%\bibitem{GJZ} F. Grunewald, A. Jaikin-Zapirain, P.A. Zalesskii
%{\em Cohomological goodness and the profinite completion of
%Bianchi groups.}  Duke Math. J.  144  (2008), 53--72. ???? was this used?

%\bibitem{hzz} SPLITTING THEOREMS FOR PRO-p GROUPS ACTING ON PRO-p TREES AND 2-GENERATED SUBGROUPS OF FREE PRO-p PRODUCTS WITH PROprocyclic AMALGAMATIONS,  WOLFGANG HERFORT, PAVEL ZALESSKII, AND THEO ZAPATA

%\bibitem{KS} "ON FINITELY GENERATED SUBGROUPS WHICH ARE OF FINITE INDEX
%IN GENERALIZEDFREE PRODUCTS" (PROCEEDINGS OF THE
%AMERICAN MATHEMATICAL SOCIETY Volume 37, Number 1, January 1973), Karas
%and Solitar

%\bibitem{King}  J. D. King, Homological finiteness conditions for pro-$p$ groups. Comm. Algebra 27 (1999), no. 10, 4969--4991

%\bibitem{Desi} D. Kochloukova,   {Subdirect products of free pro-$p$ and Demushkin groups}, Inter. J. Algebraic Computations, to appear

 %\bibitem{KS}  D. Kochloukova,H. Short, On subdirect product of free pro-$p$ groups and Demushkin groups of infinite depth, J. Algebra, 343 (2011), 160 - 172


\bibitem{KZ}  D.~Kochloukova and P.~Zalesskii, \textit{On pro-$p$ analogues of limit groups via extensions of centralizers}, Math.\ Z. {267} (2011), 109--128.

\bibitem{KZ2}  D.~Kochloukova and P.~Zalesskii, Fully residually free pro-$p$ groups, J. Algebra 324 (2010), no. 4, 782 - 792

\bibitem{melnikov} O.V. Mel'nikov, Subgroups and the homology of free products of profinite groups
Math. USSR-Izv., 34, (1990), no. 1, 97-119.




\bibitem{Ribes}
L. Ribes, On amalgamated products of profinite groups. Math. Z. 123 (1971), 357 - 364

\bibitem{RZ2}
   L. Ribes and P.A. Zalesskii,
   Pro-$p$ Trees,
   \emph{New Horizons in pro-$p$ Groups} (eds M du Sautoy, D. Segal and A. Shalev), {\rm Progress in Mathematics 184} (Birkh\"auser, Boston, 2000).


\bibitem{PavelIlir} I. Snopche,  P.A. Zalesskii  {\it Subgroup properties of pro-$p$ extensions of
centralizers.} Selecta Mathematica (to appear).

%\bibitem{Wilton} H. Wilton, {\it Hall's theorem for limit groups}, Geom. Funct. Anal. 18 (2008), no. 1,
%271--303. ???? was this used?

\bibitem{Pavel} P.A. Zalesskii, Normal subgroups of free constructions of profinite groups and
the congruence kernel in the case of positive characteristic, Izv.
Russ. Acad. Sciences, Ser Math., (1996).

\bibitem{Z-M90}   P.A. Zalesskii, O.V. Melnikov, Fundamental Groups of Graphs of Profinite Groups,
  \emph{Algebra i Analiz }{1} (1989);  {\rm translated in:} \emph{Leningrad Math. J. }{1} (1990) 921--940.

\bibitem{Theo} T. Zapata, Grupos pro-finitos limites, PhD thesis, Universidade de Bras\'ilia, 2011

\end{thebibliography}
\end{document}